\documentclass[12pt]{amsart}

\usepackage{amsfonts}
\usepackage{amssymb}
\usepackage{amsthm}
\usepackage{amsmath}
\usepackage{enumerate}
\usepackage{mathtools}
\usepackage{graphicx}
\usepackage{color}

\newtheorem{theorem}{Theorem}[section]
\newtheorem{lemma}[theorem]{Lemma}

\newtheorem{proposition}[theorem]{Proposition}

\newtheorem{question}[theorem]{Question}

\newtheorem*{theorem*}{Theorem}

\newtheorem{corollary}[theorem]{Corollary}

\newcommand{\eps}{\varepsilon}

\begin{document}
\title[Quantitative avoidance for free boundary flows]{Quantitative avoidance for free boundary flows and applications}
\author{Yueheng Bao}\address{Department of Mathematics\\ University of Toronto\\Toronto, ON M5S 2E4\\ Canada}
\email{bao624@math.utoronto.ca}

\author{Robert Haslhofer}\address{Department of Mathematics\\ University of Toronto\\Toronto, ON M5S 2E4\\ Canada}
\email{roberth@math.toronto.edu}

\begin{abstract} In this article, we introduce a new distance function between hypersurfaces with free boundary. We show that our new quantity, which we call twisted Fermi distance, is monotone under mean curvature flow with free boundary. This overcomes the stumbling block that monotonicity of the usual distance function can fail for non-convex domains, and has several applications. Most importantly, we generalize the avoidance principle for free boundary Brakke flows, recently established by the first author for convex domains, to arbitrary domains. Using this, we then show that all results from our recent joint work, including the mean-convex neighborhood theorem and the uniqueness theorem for free boundary flows through cylindrical singularities, can be generalized to arbitrary domains without any convexity assumptions as well.
\end{abstract}

\maketitle

\section{Introduction}

A fundamental property of the evolution by mean curvature is the avoidance principle, see e.g. \cite{Ilmanen_monograph}. Namely, if two mean curvature flows $M_t, N_t\subset\mathbb{R}^{n+1}$, say at least one of them compact, are disjoint at some time $t_0$, then they remain disjoint for all $t\geq t_0$. Indeed, for smooth solutions this readily follows from the classical maximum principle, but more generally the avoidance principle also holds for level set solutions, as established originally by Evans-Spruck \cite{ES} and Chen-Giga-Goto \cite{CGG}. An equivalent, but more quantitative, reformulation is that
\begin{equation}\label{dist_mon}
t\mapsto d(M_t,N_t) \quad\textrm{is monotone}.
\end{equation}
Indeed, the formulation \eqref{dist_mon} is more useful in many regards, in particular is crucial to characterize sub-solutions of the level set flow and to establish the avoidance principle for Brakke flows \cite{Ilmanen_manifold,Ilmanen_monograph}, and is the correct notion for the setting where both flows are noncompact \cite{White_avoidance}.\\

The present paper concerns (quantitative) avoidance for mean curvature flows with free boundary, namely hypersurfaces-with-boundary $M_t,N_t\subset\bar{\Omega}$ that evolve by mean curvature flow in any given domain $\Omega\subset\mathbb{R}^{n+1}$ and meet the barrier hypersurface $\partial \Omega$ perpendicularly. The avoidance principle still holds in this setting. Indeed, for smooth solutions this readily follows from the classical maximum principle and the Hopf lemma, and for level set solutions this has been established in the 90s by Giga and Sato \cite{GigaSato,Sato}, who cleverly employed conformal factors near $\partial \Omega$ to construct a suitable parabolic super 2-jet for the Neumann problem. However, these arguments are fundamentally infinitesimal in nature, and a question that remained open until now is:

\begin{question}\label{question_monotone}
Is there a monotone quantity that yields a quantitative version of the avoidance principle for free boundary flows?
\end{question}

Note that \eqref{dist_mon} clearly breaks down for free boundary flows in non-convex domains. Indeed, if $d(M_t,N_t)$ is realized by points $x\in M_t$ and $y\in N_t$ and at least one of them lies on $\partial\Omega$, then it is no longer true that difference vector $x-y$ meets the hypersurfaces perpendicularly.
In fact, while there has been a lot of recent progress on mean curvature flow with free boundary, see e.g. \cite{MizunoTonegawa,Edelen,EHIZ,Haslhofer_fb_surgery,Bao,BaoHaslhofer}, this has been a major stumbling block. Specifically, while many results are expected to hold for free boundary flows in general domains, due to the lack of answer to Question \ref{question_monotone}, they only could be proven for convex domains.\\

\subsection{Twisted Fermi distance} To answer Question \ref{question_monotone} we introduce a new distance function, which we call twisted Fermi distance. This is partly inspired, though with some differences, by fundamental work of Pacard-Ritore on the Allen-Cahn equation \cite{PacardRitore}, where they introduced twisted Fermi coordinates adapted to the Neumann condition.\\

To describe our new distance, given any smooth compact oriented free boundary flow $\mathcal{N}=\{N_t\}_{t\in [t_0,t_1]}$ in $\bar{\Omega}$, we first choose a smooth extension $\tilde{N_t}$ across $\partial \Omega$ (the results will not depend on the choice of extension). At each fixed $t$, this induces the usual Fermi coordinates $(y,z)$, where $y\in \tilde{N}_t$ and $|z|\ll 1$, simply by observing that $(y,z)\mapsto \exp_y(z\nu_{\tilde{N}_t}(y))$ is a diffeomorphism to a small tubular neighborhood of $\tilde{N}_t$. Now, fixing a small constant $\eps>0$, and a smooth cutoff function $\eta$ satisfying $\eta\equiv 1$ in $(-\infty,\eps/2)$ and $\eta\equiv 0$ in $(\eps,\infty)$, setting $\chi=\eta(d_{\partial\Omega})$, where $d_{\partial\Omega}$ denotes the distance from $\partial\Omega$, we consider the unit vector field\footnote{Note that this is different from the vector field considered by Pacard-Ritore.}
\begin{equation}\label{eq_vect_z}
Z_t=\frac{\partial_z - \langle \partial_z,\chi D d_{\partial\Omega}\rangle \chi D d_{\partial\Omega}}{|| \partial_z - \langle \partial_z,\chi D d_{\partial\Omega}\rangle \chi Dd_{\partial\Omega} ||},
\end{equation}
which by construction is tangential to $\partial \Omega$.
For $\gamma>0$ sufficiently small, the flow of the vector field $Z_t$ gives a diffeomorphism $(y,s)\mapsto \phi^{Z_t}_s(y)$ from $N_t\times (-\gamma,\gamma)$ onto a neighborhood of $N_t$ intersected with $\bar{\Omega}$.\\

In this setting, we now define the signed twisted Fermi distance of a point $x\in \bar{\Omega}$ from the free boundary flow $\mathcal{N}=\{N_t\}_{t\in [t_0,t_1]}$ in $\bar{\Omega}$ by 
\begin{equation}\label{def_d_tilde}
\tilde{d}_t(x)=\begin{cases}
s & \text{if } x=\phi^{Z_t}_s(y) \text{ for some } (y,s)\in N_t\times (-\gamma,\gamma),\\ 
\pm\gamma & \text{else}.
\end{cases}
\end{equation}
In other words, $\tilde{d}_t(x)<\gamma$ is the unique $s$, such that $x=\phi^{Z_t}_s(y)$ for some $y\in N_t$.
By the above discussion, the function $\tilde{d}_{t}$ is smooth in a neighborhood of $N_t$ intersected with $\bar{\Omega}$, specifically for $|\tilde{d}_t| < \gamma$.
Finally, if $\mathcal{M}=\{M_t\}$ is another (possibly noncompact) smooth free boundary flow in $\bar{\Omega}$, we define its twisted Fermi distance from $\mathcal{N}$ by
\begin{equation}\label{def_twisted_distance}
\tilde{d}_{\mathcal{N}}(M_t)=\inf_{x\in M_t} |\tilde{d}_t(x)|.
\end{equation}
Using these notions, we can now answer Question \ref{question_monotone} in the affirmative. We first state this in the technically simplest smooth setting:

\begin{theorem}\label{thm_monotonicity_smooth}
If $\mathcal{M}=\{M_t\}$ and $\mathcal{N}=\{N_t\}$ are smooth free boundary flows in any domain $\bar{\Omega}\subset\mathbb{R}^{n+1}$, say with $\mathcal{N}$ compact, then the twisted Fermi distance
\begin{equation}
t\mapsto \tilde{d}_{\mathcal{N}}(M_t) \quad\textrm{is monotone}.\footnote{More generally, if $\Omega$ is any domain in a Riemannian manifold with Ricci curvature bounded below by $-\kappa$, then the function $t\mapsto e^{\kappa t}\tilde{d}_{\mathcal{N}}(M_t)$ is monotone.}
\end{equation}
\end{theorem}

The theorem makes quantitative, via a novel monotone quantity, the fact that free boundary flows avoid each other. We will see that the theorem, or more precisely the differential inequalities that we establish in its proof, generalize to the setting where $\mathcal{M}$ is non-smooth. Recall from \cite{MizunoTonegawa,Edelen}, that a free boundary integral Brakke flow in $\bar{\Omega}$ is given by a family of Radon measures $\{\mu_t\}$ supported in $\bar{\Omega}$, such that
\begin{equation}\label{brakke_ineq}
\bar{D}_t\int\phi\, d\mu_t \leq  \int \left(  \partial_t\phi  + D\phi \cdot H_\ast - \phi H_\ast^2 \right)\, d\mu_t
\end{equation}
for all nonnegative test functions $\phi$ satisfying $D\phi\cdot\nu_{\partial \Omega}=0$. Here, $\bar{D}_t$ denotes the limsup of difference quotients,
\begin{equation}
H_\ast = H- 1_{\partial \Omega}(H\cdot  \nu_{\partial \Omega})\nu_{\partial\Omega},
\end{equation}
and at almost every $t$ there is an integral varifold $V_{\mu_t}$ with first variation
\begin{equation}
\delta V_{\mu_t}(X)=- \int H_\ast\cdot X\, d\mu_t\quad \textrm{for all $X$ tangential to $\partial\Omega$}.
\end{equation}
The support of $\mathcal{M}=\{\mu_t\}_{t\in I}$ is by definition the space-time set
\begin{equation}
\mathrm{spt}(\mathcal{M})=\overline{\bigcup_{t\in I} \mathrm{spt}(\mu_t)\times \{t\}}.
\end{equation}

Using these notions, we can now extend Ilmanen's avoidance principle for integral Brakke flows \cite{Ilmanen_monograph} to the free boundary setting:

\begin{theorem}\label{thm_monotonicity_singular}
Let $\mathcal{M}=\{\mu_t\}$ be a free boundary intergral Brakke flow in any domain $\bar{\Omega}\subset\mathbb{R}^{n+1}$. Then, for any smooth compact free boundary flow $\mathcal{N}=\{N_t\}_{t\in [t_0,t_1]}$, choosing $\gamma=\gamma(\mathcal{N})>0$ sufficiently small,
\begin{equation}
t\mapsto \tilde{d}_{\mathcal{N}}\left((\mathrm{spt}\mathcal{M})_t\right)\quad\textrm{is monotone}.
\end{equation}
In particular, we have the free boundary avoidance principle
\begin{equation}
(\mathrm{spt}\mathcal{M})_{t_0}\cap N_{t_0}=\emptyset \,\,\Rightarrow\,\, (\mathrm{spt}\mathcal{M})_t\cap N_t=\emptyset,\, \forall t\geq t_0.
\end{equation}
\end{theorem}

This result generalizes \cite[Theorem 1.5]{Bao} in two ways, first by making the avoidance quantitative, and second -- more crucially -- by removing the formidable convexity assumption on $\Omega$. Moreover, since subsolutions of the free boundary level set flow avoid each other by a result of Giga-Sato \cite{GigaSato}, our theorem also implies the following corollary.

\begin{corollary}
Let $\mathcal{M}, \mathcal{N}$ be a free boundary intergral Brakke flow in any domain $\bar{\Omega}\subset\mathbb{R}^{n+1}$, say with $\mathcal{N}$ compact. Then, we have
\begin{equation}
(\mathrm{spt}\mathcal{M})_{t_0}\cap (\mathrm{spt}\mathcal{N})_{t_0}=\emptyset \,\,\Rightarrow\, \, (\mathrm{spt}\mathcal{M})_t\cap (\mathrm{spt}\mathcal{N})_{t}=\emptyset,\,\forall t\geq t_0.
\end{equation}
\end{corollary}
This generalizes another result of the first author \cite[Corollary 1.6]{Bao} again by removing the formidable convexity assumption on $\Omega$.\\

\subsection{Applications} Applying our new monotone twisted Fermi distance, we can now generalize many results for free boundary flows to arbitrary domains $\Omega$, thus removing the restrictive convexity assumption, which -- as discussed above -- has been a major stumbling block in developing a satisfying theory of flows with free boundary.\\

Our first application -- the avoidance principle for free boundary integral Brakke flows in general domains -- has been discussed above.\\

Second, also regarding foundational results for weak solutions, we can now establish the existence of matching free boundary integral Brakke flows for the inner/outer free boundary flow. To state this, given any smooth compact free boundary hypersurface $M\subset\bar{\Omega}$ as usual we consider the enclosed $K,K'\subset\bar{\Omega}$ enclosed by $M$, such that
\begin{equation}\label{encl_regions}
K\cup K'=\bar{\Omega}\quad\mathrm{and}\quad K\cap K'=M.
\end{equation}
Denoting by $\mathcal{K}$ and $\mathcal{K}'$ the space-time tracks of their free boundary level set flows, c.f. \cite{GigaSato,Sato}, the outer and inner free boundary flow have then be defined in \cite{Bao}, generalizing the notions from \cite{HershkovitsWhite}, as\footnote{Here, $\partial \mathcal{X}=\mathcal{X}\setminus \mathrm{Int}_{\bar{\Omega}\times \mathbb{R}}(\mathcal{X})$, where $\mathrm{Int}_{\bar{\Omega}\times \mathbb{R}}(\mathcal{X})$ denotes the interior of $\mathcal{X}\subseteq\bar{\Omega}\times\mathbb{R}$.}
\begin{equation}
M_t= \{ x\in \bar{\Omega} \, | \, (x,t)\in \partial \mathcal{K}\},\quad M_t'= \{ x\in \bar{\Omega} \, | \, (x,t)\in \partial \mathcal{K}'\}.
\end{equation} 

\begin{theorem}\label{cor_matching}
Let $M$ be a smooth compact free boundary hypersurface in a general domain $\bar{\Omega}\subset\mathbb{R}^{n+1}$. Then, there exist unit-regular, cyclic free boundary integral Brakke flows $\{\mu_t\}_{t\geq 0}$ and $\{\mu_t'\}_{t\geq 0}$ starting from $\mathcal{H}^n\lfloor M$, whose support is given by the outer/inner free boundary flow.
\end{theorem}

This generalized \cite[Theorem 1.8]{Bao} by removing the convexity assumption on $\Omega$. Similarly, \cite[Corollary 5.5]{Bao} can be generalized as well.\\

To discuss our third application, recall that the outer free boundary flow has a half-cylindrical singularity at $X_0=(x_0,t_0)\in \partial\Omega\times \mathbb{R}_+$ if for some $\lambda_i\to \infty$ the parabolically rescaled flows $\mathcal{D}_{\lambda_i}(\mathcal{K}-X_0)$ converge locally smoothly with multiplicity-one up to rotation to either
\begin{equation}
\left\{\mathbb{R}^k_{+} \times \bar{B}^{n+1-k}(\sqrt{2(n-k)|t|})\right\}_{t\leq 0},
\end{equation}
or
\begin{equation}
\left\{\bar{B}_{+}^{n+1-k}(\sqrt{2(n-k)|t|})\times\mathbb{R}^k \right\}_{t\leq 0}.
\end{equation}

\begin{theorem}\label{thm_mcn}If the outer or inner free boundary flow starting from any smooth compact free boundary hypersurface in a general domain $\bar{\Omega}\subset\mathbb{R}^{n+1}$ has a (half-)cylindrical singularity at $X_0=(x_0,t_0)\in \bar{\Omega}\times \mathbb{R}_+$, then the flow is mean-convex in a space-time neighborhood of $X_0$.
\end{theorem}

This generalized the mean-convex neighborhood theorem for free boundary flows from our joint work \cite[Theorem 1.2]{BaoHaslhofer} again by removing the formidable convexity assumption on $\Omega$. Similarly, the canonical neighborhood theorem from \cite[Theorem 1.3]{BaoHaslhofer} generalizes as well.\\

Our fourth application concerns nonfattening, which is implied by establishing a lower bound on the discrepancy time
\begin{equation}
T_{\mathrm{disc}}= \inf \{ t > 0 \, | \, M_t\neq  M_t'  \}.
\end{equation}

\begin{theorem}\label{thm_nonfatt}For general domains $\bar{\Omega}$,
if $0<T\leq T_{\mathrm{disc}}$ and all singularities at time $T$ have a mean-convex neighborhood for the inner or outer free boundary flow, then $T<T_{\mathrm{disc}}$.
\end{theorem}

This generalizes \cite[Theorem 1.4]{BaoHaslhofer}, again by removing the convexity assumption. In the formulation of the theorem, one has to be careful with the definition of singularities due to the pop-up phenomenon in general domains as observed by Edelen \cite{Edelen}. Specifically, in Theorem \ref{thm_nonfatt} the singular set is defined as the complement of the two-sided regular points.
Finally, combining this with Theorem \ref{thm_mcn} we conclude:

\begin{corollary}\label{cor_uniqueness}
Free boundary flow through singularities in general domains is well-posed as long as all singularities are (half)-cylindrical.
\end{corollary}

This generalizes \cite[Corollary 1.5]{BaoHaslhofer}, once again by removing the formidable convexity assumption, and thus gives a satisfying theory for free boundary flow through cylindrical and half-cylindrical singularities.\\

\subsection{Outline}
As common with monotonicity formulas, the key is to come up with the correct quantity and formulation, and once this is done the proofs are very short and elementary.
Specifically, we first observe that the mean curvature flow equation for $\mathcal{N}$ together with a Laplace comparison-type argument yield
\begin{equation}
(\partial_t-\Delta_{\Omega})\tilde{d}_t(x)\geq 0,\quad \textrm{whenever } 0<\tilde{d}_t(x)<\gamma.
\end{equation}
On the other hand, using the mean curvature flow equation for $\mathcal{M}$, we infer that
\begin{equation}
\left(\frac{d}{dt}-\Delta_{M_t}\right) \tilde{d}_t = (\partial_t-\Delta_{\Omega} ) \tilde{d}_t +D^2 \tilde{d}_t (\nu_{M_t},\nu_{M_t}).
\end{equation}
Combining these observations and the maximum principle, taking also into account the boundary identity $D\tilde{d}_t\cdot\nu_{\partial \Omega}=0$, we show that this yields the monotonicity of $\tilde{d}_{\mathcal{N}}(M_t)$ in the smooth setting. Generalizing this to the non-smooth setting involves some additional steps, in particular establishing a half-ball barrier lemma and plugging a suitable smeared version of $\tilde{d}_t$ into the free boundary Brakke inequality, but also these steps are neither long nor difficult. Finally, as we will explain, having established monotonicity the applications follow by suitably adjusting the arguments from \cite{Bao} and \cite{BaoHaslhofer}.\\

\bigskip

{\bf Acknowledgements:} Y.B. has been supported by a Blyth Fellowship, a Mary H. Beatty Fellowship, a Department of Mathematics Graduate Program Award and an International Graduate Student Scholarship from the University of Toronto.
R.H. has been supported by the NSERC Discovery Grant RGPIN-2023-04419.\\

\section{Monotonicity of the twisted Fermi distance}\label{sec_monotonicity}

Given a smooth compact oriented free boundary flow $\mathcal{N}=\{N_t\}_{t\in [t_0,t_1]}$ in $\bar{\Omega}$, we fix $\gamma=\gamma(\mathcal{N})>0$ sufficiently small, and consider the signed twisted Fermi distance $\tilde{d}_t$ from $\mathcal{N}$ as defined in \eqref{def_d_tilde}.

\begin{proposition}\label{prop_evol_twisted}
The signed twisted Fermi distance satisfies
\begin{equation}\label{prop_twist_f1}
(\partial_t-\Delta_{\Omega})\tilde{d}_t\geq 0\quad \textrm{in } \Omega\cap \{ 0< \tilde{d}_t < \gamma\},
\end{equation}
and \begin{equation}\label{prop_twist_f2}
D \tilde{d}_t \cdot \nu_{\partial \Omega}=0 \quad \textrm{on } \partial\Omega\cap \{ | \tilde{d}_t| < \gamma\}.
\end{equation}
\end{proposition}

\begin{proof}
If $0< \tilde{d}_t(x) < \gamma$ is attained at $y\in N_t$, then by the mean curvature flow equation for $\mathcal{N}$ we have
\begin{equation}
\partial_t \tilde{d}_t(x) =-H_{N_t}(y),
\end{equation}
where we use the sign convention that $H_{N_t}$ denotes the mean curvature of $N_t=\{  s=0 \}$ with respect to the unit normal $\nu_{N_t}$. On the other hand, it is a well known fact (see e.g. \cite[Lemma 10.5]{PacardRitore}) that the mean curvature of level sets is given by the Laplacian, namely
\begin{equation}
\Delta_{\Omega} \tilde{d}_t(x) = -H_{\{s=\tilde{d}_t(x)\}}(x).
\end{equation}
Combining these two equations gives
\begin{equation}\label{eq_hs0}
(\partial_t-\Delta_{\Omega})\tilde{d}_t(x)=H_{\{s=\tilde{d}_t(x)\}}(x)-H_{\{s=0\}}(y).
\end{equation}
Now, since our vector field $Z_t$ defined in \eqref{eq_vect_z} has unit norm, we have $|D \tilde{d}_t|=1$, and consequently the Bochner formula gives
\begin{equation}
0=\tfrac{1}{2}\Delta_{\Omega} |D \tilde{d}_t|^2=D_{D \tilde{d}_t} \Delta_{\Omega} \tilde{d}_t + |D^2 \tilde{d}_t|^2.
\end{equation}
Considering the integral curve $\gamma_\sigma= \phi^{Z_t}_\sigma(y)$ from $y$ to $x$ this yields
\begin{equation}
\frac{d}{d\sigma} H_{\{s=\sigma\}}(\gamma_\sigma) =  |D^2 \tilde{d}_t(\gamma_\sigma)|^2\geq 0.
\end{equation}
Remembering \eqref{eq_hs0}, this proves \eqref{prop_twist_f1}. Finally, \eqref{prop_twist_f2} follows directly from the fact that our vector field $Z_t$ is tangential to $\partial \Omega$.
\end{proof}

We can now prove Theorem \ref{thm_monotonicity_smooth}, which we restate for convenience:

\begin{theorem}
If $\mathcal{M}=\{M_t\}$ and $\mathcal{N}=\{N_t\}$ are smooth free boundary flows in any domain $\bar{\Omega}\subset\mathbb{R}^{n+1}$, say with $\mathcal{N}$ compact, then fixing $\gamma=\gamma(\mathcal{N})>0$ small enough, the twisted Fermi distance from \eqref{def_twisted_distance} satisfies
\begin{equation}
t\mapsto \tilde{d}_{\mathcal{N}}(M_t) \quad\textrm{is monotone}.
\end{equation}
\end{theorem}

\begin{proof}
Note that it suffices to show that at any $t$ with $0<\tilde{d}_{\mathcal{N}}(M_t) <\gamma$ we have $\tfrac{d}{dt} \tilde{d}_{\mathcal{N}}(M_t) \geq 0$. Moreover, possibly after flipping the orientation of $\mathcal{N}$, we can assume that, given such $t$, the infimum in \eqref{def_twisted_distance} is attained at some $x\in M_t$ with $\tilde{d}_t(x)=+\tilde{d}_{\mathcal{N}}(M_t)$. Now, since $\mathcal{M}$ moves by mean curvature flow, the chain rule gives
\begin{equation}
\frac{d}{dt} \tilde{d}_t(x) = \partial_t \tilde{d}_t(x) + D  \tilde{d}_t(x) \cdot H_{M_t}(x).
\end{equation}
Together with the formula $\Delta_{M_t} f=\mathrm{div}_{M_t} Df+Df\cdot H$ this yields
\begin{equation}
\left(\frac{d}{dt}-\Delta_{M_t}\right) \tilde{d}_t = (\partial_t-\Delta_{\Omega} ) \tilde{d}_t +D^2 \tilde{d}_t (\nu_{M_t},\nu_{M_t}).
\end{equation}
Since $\tilde{d}_t$ attains its minimum at $x\in M_t$, we have $\nu_{M_t}(x)=D \tilde{d}_t(x)$, and thus differentiating $|D\tilde{d}|=1$ we see that $D^2 \tilde{d}_t (\nu_{M_t},\nu_{M_t})$ vanishes at $x$. Hence, together with Proposition \ref{prop_evol_twisted}, the maximum principle implies the assertion.
\end{proof}

To generalizing the above arguments to the non-smooth setting, we start with the following (half-)ball avoidance lemma.

\begin{lemma}\label{lemma_halfball}
For every $x_0\in\bar{\Omega}$ and $\lambda\in(0,1)$, there are $\rho_0>0$ and
$c_0>0$ with the following significance. If $0<\rho<\rho_0$, and $\mathcal{M}$ is a free boundary integral Brakke flow in $\bar{\Omega}$ with
$(\mathrm{spt}\mathcal{M})_{t_0}\cap B_\rho(x_0)=\emptyset$, then
\begin{equation}\label{eq:continuous-ball-avoidance}
(\mathrm{spt}\mathcal{M})_{t}\cap B_{(1-\lambda)\rho}(x_0)=\emptyset
  \qquad
  \text{for }t_0\leq t\leq t_0+c_0\rho^2.
\end{equation}
\end{lemma}

\begin{proof}
For convex $\Omega$ this has already been shown in \cite[Lemma 4.4]{Bao}, and in fact the convexity assumption has only been used to ensure that
\begin{equation}
(x-x_0)\cdot \nu_{\partial \Omega}(x) \geq 0 \quad \text{for }x\in\partial\Omega\cap B_\rho(\hat x_0).
\end{equation}
Now, for general $\Omega$, if $d(x_0,\partial\Omega)\leq \rho$ we will shrink and shift inwards the ball slightly. Specifically,
 letting $y_0$ be the nearest point of $\partial\Omega$ to $x_0$, and denoting by $\nu_{\partial\Omega}(y_0)$ the outward unit normal as usual, set
\begin{equation}
  \varepsilon=\frac{\lambda\rho}{4},
  \qquad
  \hat x_0=x_0-\varepsilon\nu_{\partial\Omega}(y_0),
  \qquad
  \hat{\rho}=\rho-\varepsilon.
\end{equation}
Then, $B_{\hat \rho}(\hat x_0)\subseteq B_\rho(x_0)$, and the $C^2$-graph representation
of $\partial\Omega$ gives
\begin{equation}\label{eq:shifted-center-sign}
  (x-\hat x_0)\cdot\nu_{\partial\Omega}(x)
\geq0
  \quad
  \text{for }x\in\partial\Omega\cap B_{\hat \rho}(\hat x_0),
\end{equation}
provided $\rho_0$ is small enough. Hence, considering the standard test function
\begin{equation}
  \phi(x,t)=
  \left(
    1-\frac{|x-\hat x_0|^2+2n(t-t_0)}{{\hat \rho}^2}
  \right)_+^3
\end{equation}
we can conclude the proof similarly as in \cite[Lemma 4.4]{Bao}.
\end{proof}

We can now prove Theorem \ref{thm_monotonicity_singular}, which we restate for convenience:

\begin{theorem}
Let $\mathcal{M}=\{\mu_t\}$ be a free boundary intergral Brakke flow in any domain $\bar{\Omega}\subset\mathbb{R}^{n+1}$. Then, for any smooth compact free boundary flow $\mathcal{N}=\{N_t\}_{t\in [t_0,t_1]}$, choosing $\gamma=\gamma(\mathcal{N})>0$ sufficiently small,
\begin{equation}
t\mapsto \tilde{d}_{\mathcal{N}}\left((\mathrm{spt}\mathcal{M})_t\right)\quad\textrm{is monotone}.
\end{equation}
In particular, $\mathrm{spt}\mathcal{M}$ satisfies the free boundary avoidance principle.
\end{theorem}

\begin{proof}Fix $\gamma=\gamma(\mathcal{N})>0$ sufficiently small. Given any $\delta>0$, with
\begin{equation}
\delta<\min\{\gamma, \tilde{d}_{\mathcal{N}}\left((\mathrm{spt}\mathcal{M})_{t_0}\right)\},
\end{equation}
motivated by \cite{Ilmanen_monograph,HershkovitsWhite_avoidance,Bao} we consider the test-function
\begin{equation}
\phi(x,t)=h(|\tilde{d}_t(x)|),
\end{equation}
where $h$ is a smooth monotone function, such that $h(r)=(\delta-r)_+^3$ for $r\geq \delta/2$ and $h(r)=\delta^3$ for $r\leq \delta/4$. Plugging this into the Brakke inequality \eqref{brakke_ineq} and using integration by parts, we infer that
\begin{equation}
\bar{D}_t\int\phi\, d\mu_t \leq  \int ( \partial_t  - \mathrm{div}_{V_{\mu_t}}D) \phi\, d\mu_t,
\end{equation}
where there is no boundary term thanks to $D\phi\cdot \nu_{\partial \Omega}=0$, c.f. \cite[Proposition 3.2]{Edelen}. To proceed, observe that by definition of the signed twisted Fermi distance we have $0.9 d_t\leq \tilde{d}_t\leq 1.1 d_t$ on $\mathrm{spt}(\phi)$ provided $\gamma$ is small enough. Hence, by Lemma \ref{lemma_halfball} it suffices to consider times $t$ with $ \tilde{d}_{\mathcal{N}}\left((\mathrm{spt}\mathcal{M})_t\right)\geq \delta/2$. Then, on $\mathrm{spt}(\mu_t)\cap \{\tilde{d}_t\geq 0\}$ we have
\begin{align}\label{equ:operphi}
(\partial_t-\operatorname{div}_{V_{\mu_t}}D)\phi &=(\partial_t-\Delta_{\Omega})\phi+D^2\phi(\nu_{V_{\mu_t}},\nu_{V_{\mu_t}})\nonumber\\
&\leq -3 s^2 D^2 \tilde{d}_t(\nu_{V_{\mu_t}}, \nu_{V_{\mu_t}})-6 s\big(1-(\nu_{V_{\mu_t}} \cdot D \tilde{d}_t)^2\big),
\end{align}
where we used Proposition \ref{prop_evol_twisted} and abbreviated $s=(\delta-|\tilde{d}_t|)_+$. Moreover, since $D^2 \tilde{d}_t$ is a finite dimensional quadratic form that vanishes in
direction $D \tilde{d}_t$, as can be seen by differentiating $|D\tilde{d}_t|^2 = 1$, we have
\begin{equation}
|D^2 \tilde{d}_t(\nu_{V_{\mu_t}}, \nu_{V_{\mu_t}})|\leq C\big(1-(\nu_{V_{\mu_t}} \cdot D \tilde{d}_t)^2\big).
\end{equation}
Since $3Cs^2\leq 6s$ for $\gamma$ small enough, this yields
\begin{equation}
\bar{D}_t\int\phi\, d\mu_t \leq 0.
\end{equation}
This implies the assertion.
\end{proof}

\bigskip

\section{Applications to flows through singularities}\label{sec_applications}

In this short final section, we explain how the applications follow.

\begin{proof}[Proof of Theorem \ref{cor_matching}]
In the matching Brakke flow proof in \cite[Theorem 1.8]{Bao} convexity of $\Omega$ only entered in the following two ways:
\begin{enumerate}[(i)]
\item Via \cite[Lemma 5.4]{Bao}, which for any smooth compact free boundary hypersurface $M\subset\bar{\Omega}$ constructs a nontrivial one-parameter family $\{M^s\}_{s=(-\eps,\eps)}$ of free boundary hypersurfaces with $M^0=M$.
\item Via the avoidance principle for free boundary integral Brakke flows from \cite[Theorem 1.5]{Bao}, which assumed convexity.
\end{enumerate}
Using our twisted Fermi distance we can now prove \cite[Lemma 5.4]{Bao} by simply setting $M^s=\{ \tilde{d}=s\}$, which is much shorter and works for general domains without any convexity assumption. Hence, the corollary follows by using Theorem \ref{thm_monotonicity_singular} in lieu of \cite[Theorem 1.5]{Bao}.
\end{proof}

\begin{proof}[Proof of Theorem \ref{thm_mcn}]
Using Corollary \ref{cor_matching} in lieu of \cite[Theorem 1.8]{Bao}, and Theorem \ref{thm_monotonicity_singular} in lieu of \cite[Theorem 1.5]{Bao}, the proof of the mean-convex neighborhood theorem from \cite[Theorem 1.2]{BaoHaslhofer} goes through verbatim for general domains.
\end{proof}

\begin{proof}[Proof of Theorem \ref{thm_nonfatt}]
In the proof of the nonfattening theorem in \cite[Theorem 1.4]{BaoHaslhofer} we considered the function
$w=\varphi f + (1-\varphi)d_t$,
which interpolates between the arrival time function $f$ near the singularities and the signed distance $d_t$ from $M_t$ in the smooth part. Here, $\varphi$ is a suitable cutoff function that is identically $1$ in a neighborhood of the singularities. The convexity assumption on the domain was used to ensure that $\pm Dd_t\cdot \nu_{\partial\Omega}\geq 0$ when $\pm d_t\geq 0$.
We now replace $d_t$ by the signed twisted Fermi distance $\tilde{d}_t$. Observing that $\tilde{d}_t$ is smooth in a neighborhood of the bounded curvature part with the estimate $|D\tilde{d}_t - Dd_t|<\eps$, and satisfies $D\tilde{d}_t\cdot \nu_{\partial\Omega}=0$ on $\partial\Omega$, the argument from the proof of \cite[Theorem 1.4]{BaoHaslhofer}  goes through for general domains.
\end{proof}

\bigskip

\bibliography{BaoHaslhofer_fb_avoidance}

\bibliographystyle{abbrv}

\end{document}